\documentclass[12pt,A4paper]{article}

\usepackage[left=32mm,right=32mm,top=30mm,bottom=32mm]{geometry}
\usepackage{labelfig}
\usepackage{epsfig}
\usepackage{epstopdf}

\usepackage{xfrac}

\usepackage{color}
\usepackage{amsthm,amsmath,amssymb}
\usepackage{booktabs}
\usepackage{txfonts}
\usepackage{mathpazo}
\usepackage{microtype}
\usepackage{overpic}
\usepackage{bm}
\usepackage{sectsty}
\usepackage[
	pdftitle={PDFTitle},
	pdfauthor={Hugo Parlier and Lionel Pournin},
	ocgcolorlinks,
	linkcolor=linkred,
	citecolor=linkred,
	urlcolor=linkblue]
{hyperref}

\definecolor{linkred}{rgb}{0.48,0.1,0.05}
\definecolor{linkblue}{RGB}{16, 78, 139}

\usepackage[hang,flushmargin]{footmisc}
\usepackage{enumitem}
\usepackage{titlesec}
	\titlespacing{\section}{0pt}{12pt}{0pt}
	\titlespacing{\subsection}{0pt}{6pt}{0pt}
	
\titlelabel{\thetitle.\quad}

\makeatletter 

\long\def\@footnotetext#1{% 
\H@@footnotetext{% 
\ifHy@nesting 
\hyper@@anchor{\@currentHref}{#1}% 
\else 
\Hy@raisedlink{\hyper@@anchor{\@currentHref}{\relax}}#1% 
\fi 
}}

\def\@footnotemark{% 
\leavevmode 
\ifhmode\edef\@x@sf{\the\spacefactor}\nobreak\fi 
\H@refstepcounter{Hfootnote}% 
\hyper@makecurrent{Hfootnote}% 
\hyper@linkstart{link}{\@currentHref}% 
\@makefnmark 
\hyper@linkend 
\ifhmode\spacefactor\@x@sf\fi 
\relax 
}% 

\ifFN@multiplefootnote% 
\renewcommand*\@footnotemark{% 
\leavevmode 
\ifhmode 
\edef\@x@sf{\the\spacefactor}% 
\FN@mf@check 
\nobreak 
\fi 
\H@refstepcounter{Hfootnote}% 
\hyper@makecurrent{Hfootnote}% 
\hyper@linkstart{link}{\@currentHref}% 
\@makefnmark 
\hyper@linkend 
\ifFN@pp@towrite 
\FN@pp@writetemp 
\FN@pp@towritefalse 
\fi 
\FN@mf@prepare 
\ifhmode\spacefactor\@x@sf\fi 
\relax% 
}% 
\fi 

\makeatother 

\theoremstyle{plain}
\newtheorem{theorem}{Theorem}[section]

\newtheorem{lemma}[theorem]{Lemma}
\newtheorem{corollary}[theorem]{Corollary}

\theoremstyle{definition}

\newcommand{\diam}{{\rm diam}}

\newcommand{\contract}{\mathord{\varparallelinv}}

\newcommand{\FS}{\mathcal F {(\Sigma_n)}}

\newcommand{\MF}{\mathcal M \mathcal F}
\newcommand{\MFS}{{\mathcal M \mathcal F} (\Sigma_n)}

\newcommand{\mcg}{{\rm Mod}}

\sectionfont{\large \sc}
\subsectionfont{\normalsize}

\long\def\symbolfootnote[#1]#2{\begingroup%
\def\thefootnote{\fnsymbol{footnote}}\footnote[#1]{#2}\endgroup}

\def\blfootnote{\xdef\@thefnmark{}\@footnotetext}

\begin{document}

{\Large \bfseries \sc Modular flip-graphs of one holed surfaces}

{\bfseries Hugo Parlier\symbolfootnote[1]{\normalsize Research supported by Swiss National Science Foundation grants PP00P2\textunderscore 128557 and PP00P2\textunderscore153024
}, Lionel Pournin\symbolfootnote[2]{\normalsize Research funded by Ville de Paris {\'E}mergences project ``Combinatoire {\`a} Paris''.\\
%{\em 2010 Mathematics Subject Classification:} Primary: . Secondary: . \\
}}

\vspace{0.2cm}
{\em Abstract.}
We study flip-graphs of triangulations on topological surfaces where distance is measured by counting the number of necessary flip operations between two triangulations. We focus on surfaces of positive genus $g$ with a single boundary curve and $n$ marked points on this curve; we consider triangulations up to homeomorphism with the marked points as their vertices. Our main results are upper and lower bounds on the maximal distance between triangulations depending on $n$ and can be thought of as bounds on the diameter of flip-graphs up to the quotient of underlying homeomorphism groups. The main results assert that the diameter of these quotient graphs grows at least like $5n/2$ for all $g\geq 1$. Our upper bounds grow at most like $[4 -1/(4g)]n$ for $g\geq 2$, and at most like $23n/8 $ for the torus.
\section{Introduction}\label{sec:intro}

The set of triangulations of a given surface can be given a geometry by defining the distance between two triangulations as the necessary number of flip operations needed to transform one of them into the other. As stated, this definition is somewhat vague, but if the surface is a (Euclidean) polygon and we think of triangulations as being geometric (realized by line segments), then a flip consists in removing an edge from a triangulation and replacing it by the unique other edge so that the result is still a triangulation of the polygon. Given this geometry, the set of triangulations of a polygon is well studied: it is the graph of the associahedron \cite{Lee1989, Stasheff1963, Stasheff2012, Tamari1954}. In particular its diameter is $2n-10$ for $n>12$ \cite{Pournin2014, SleatorTarjanThurston1988}.

The graph of the associahedron can be defined in purely topological terms: a polygon is a topological disk with $n$ marked points on its boundary and a triangulation is a maximal collection of disjoint isotopy classes of arcs with endpoints in the marked points. (The isotopies are relative to the endpoints and by disjoint it is meant disjoint in their interiors.) Flip transformations are topological; one way of defining them is that two triangulations are related by a flip if they differ only in one edge. This topological description works when one replaces the disk by any finite type surface with marked points, provided there are marked points on boundary curves. However, provided the surface has enough topology (for instance if it has positive genus), the flip-graph is infinite and has infinite diameter. The mapping class group (self-homeomorphisms up to isotopy) of the underlying surface acts nicely on it: indeed it is basically the isomorphism group of the graph (see Section \ref{sec:prelim} for precise definitions, statements, and references). As such, we can quotient flip-graphs to obtain well defined finite graphs which represent all possible types of triangulations and we can measure distances between them. We call these graphs {\it modular flip-graphs} and we are interested in their geometry. 

In a previous paper \cite{ParlierPournin2014}, we explored the diameter of these graphs for \emph{filling surfaces}: these surfaces have a privileged boundary curve, but otherwise arbitrary topology. In particular, they were allowed to have interior marked points, more than one boundary, and arbitrary genus. Here we focus our attention to the special case of \emph{one holed surfaces} that have a unique boundary curve and no interior marked points. We investigate the growth of the diameter of the corresponding modular flip-graphs in function of the number of marked points on the boundary, while the genus is fixed. These points are \emph{labelled}: we quotient by homeomorphisms that leave them individually fixed. 

The case that we focus on most is the torus; it provides a natural variant on the case of the disk and, as we shall see, is already quite intriguing. For this surface we are able to prove the following.

\begin{theorem}\label{thm:torus}
Let $\Sigma_n$ be a one holed torus with $n$ marked points on the boundary and let $\mathcal{MF}(\Sigma_n)$ be its modular flip-graph. Then
$$
%\left\lfloor{ \frac{5}{2} n }\right\rfloor+\diam(\MF(\Sigma_1))-2 \leq  \diam(\mathcal{MF}(\Sigma_n)) \leq \frac{23}{8}  n + 8
\left\lfloor{ \frac{5}{2} n }\right\rfloor-2 \leq  \diam(\mathcal{MF}(\Sigma_n)) \leq \frac{23}{8}  n + 8\mbox{.}
$$
\end{theorem}

There is a clear gap between our upper and lower bounds (on the order of $3n/8$ for large $n$) that we are unable to close; in fact it is even tricky to guess what the correct growth rate might be. It could very well be that the lower bounds can be improved but that the necessary combinatorics are currently out of our reach or, on the contrary, it might be that the upper bounds are considerably improvable (or both!). We point out that, in the instances where matching upper and lower bounds are known \cite{ParlierPournin2014, Pournin2014, SleatorTarjanThurston1988} the lower bounds have always been the more difficult ones to obtain. 

The methods of the above theorem can be used to deduce more general results about surfaces of genus $g\geq 1$.

\begin{theorem}\label{thm:genusg}
Let $\Sigma_n$ be a one holed surface of genus $g\geq 1$ with $n$ marked points on the boundary and let $\mathcal{MF}(\Sigma_n)$ be its modular flip-graph. Then
$$
%\left\lfloor{ \frac{5}{2} n }\right\rfloor+\diam(\MF(\Sigma_1))-2 \leq  \diam(\mathcal{MF}(\Sigma_n)) \leq (4 - \frac{1}{4g})  n + C_g
\left\lfloor{ \frac{5}{2} n }\right\rfloor-2 \leq  \diam(\mathcal{MF}(\Sigma_n)) \leq \left(4 - \frac{1}{4g}\right)  n + K_g\mbox{,}
$$
where $K_g$ only depends on $g$. 
\end{theorem}
 
As a direct consequence of this result and the results of \cite{ParlierPournin2014}, we obtain the following.
\begin{corollary}
Let $\Sigma_n$ be a filling surface with fixed topology and $n$ marked points on the privileged boundary. If this surface is not homeomorphic to a disk or a once punctured disk, then the diameter of its modular flip-graph grows at least as $5n/2$. 
\end{corollary}

\section{Preliminaries}\label{sec:prelim}

We begin this section by defining and describing the objects we are interested in.

Our basic setup is as follows. Consider an orientable topological surface of finite type $\Sigma$ with a single boundary curve. It has no marked points on it  but will be endowed with them in what follows. We'll denote by $g\geq 0$ the genus of $\Sigma$ (so if $g=0$ then $\Sigma$ is a disk). 

For any positive integer $n$, from $\Sigma$ we obtain a surface $\Sigma_n$ by placing $n$ marked points on the privileged boundary of $\Sigma$. These marked points are {\it labelled} from $a_1$ to $a_n$, clockwise around the privileged boundary. We are interested in triangulating $\Sigma_n$ and studying the geometry of the resulting flip-graphs. An {\it arc} of $\Sigma_n$ is an {\it essential} isotopy class of non-oriented simple paths between two marked points (not necessarily distinct). The isotopies we consider are relative to endpoints and by essential we mean not isotopic to a single marked point. In particular the boundary curve is cut into $n$ essential arcs between subsequent marked points: we call these arcs boundary arcs and denote by $\alpha_p$ the boundary arc between $a_p$ and the next marked point clockwise around the privileged boundary. All other arcs are {\it interior} arcs but unless specifically stated, by arc we'll generally mean interior arc.

 A {\it triangulation} of $\Sigma_n$ is a maximal collection of arcs that can be realized in such a way that two of them do not cross. For fixed $\Sigma_n$, we call the number of interior arcs of a triangulation the {\it arc complexity} of $\Sigma_n$. Note that by an Euler characteristic argument, it increases linearly in $n$. Although triangulations are not necessarily ``proper" triangulations in the usual sense (triangles can share multiple arcs and arcs can be loops), they do cut the surface into a collection of triangles. This is one reason for using the terminology of {\it arc} of a triangulation instead of an {\it edge}, the other being to avoid confusion between the edges of a triangulation and the edges of the flip-graphs we'll now define.
 
The {\it flip-graph} $\FS$ is constructed as follows. Vertices of $\FS$ are the triangulations of $\Sigma_n$ and two vertices share an edge if they coincide in all but one arc. Equivalently they share an edge if they are related by a {\it flip}, as shown on Fig. \ref{fig:flip}.
\begin{figure}
\begin{centering}
\includegraphics{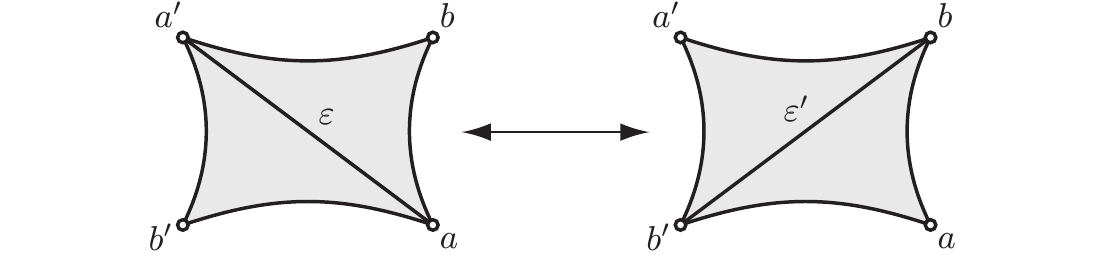}
\caption{The flip that exchanges arcs $\varepsilon$ and $\varepsilon'$}
\label{fig:flip}
\end{centering}
\end{figure}
This graph is finite if and only if $\Sigma$ is of genus $0$ but it is always connected and any isotopy class of arc can be introduced into a triangulation by a finite number of flips (see for instance \cite{Mosher1988}).

In the event that $\FS$ is infinite, there is a non-trivial natural action of the group of self-homeomorphisms of $\Sigma_n$ on $\FS$. We require that homeomorphisms fix all marked points individually. We denote $\mcg(\Sigma_n)$ the group of such homeomorphisms up to isotopy. Note that once $n\geq 3$, by the action on the boundary of $\Sigma$, all such homeomorphisms are necessarily orientation preserving. As we are interested primarily in large $n$, we do not need to worry about orientation reversing homeomorphisms. We further note that $\mcg(\Sigma_n)$ is almost the automorphism group of $\FS$ and had we allowed homeomorphisms to exchange marked points, it would have been exactly the automorphism group \cite{KorkmazPapadopoulos2012}. As we do not allow it, the automorphism group of the resulting graphs is the cyclic group on $n$ elements. 

The combinatorial moduli spaces (the modular flip-graphs) we are interested in are thus 
$$
\MFS = \FS / \mcg (\Sigma_n).
$$
As described above, these are connected finite graphs which we think of as discrete metric spaces where distance comes from assigning length $1$ to each edge. In particular, some of these graphs have loops (a single edge from a vertex to itself) but adding or removing a loop gives rise to an identical metric space on the set of vertices so we think of these graphs as not having any loops. Our main focus is on the vertex diameter of these graphs, which we denote $\diam(\MFS)$, and how these grow in function of $n$ for fixed $\Sigma$.

Our lower bounds on $\diam(\MFS)$ will be obtained using the operation of deleting a vertex from triangulations of $\Sigma_n$ already used in \cite{ParlierPournin2014,Pournin2014}. (Equivalently one can think of this operation as the contraction of a boundary arc.) Assuming that $n$ is greater than $1$, these operations induce graph homomorphisms from $\MFS$ to $\MF(\Sigma_{n-1})$. Consider a triangulation $T$ in $\MFS$ and a vertex $a_p$ of $T$.
\begin{figure}[b]
\begin{centering}
\includegraphics{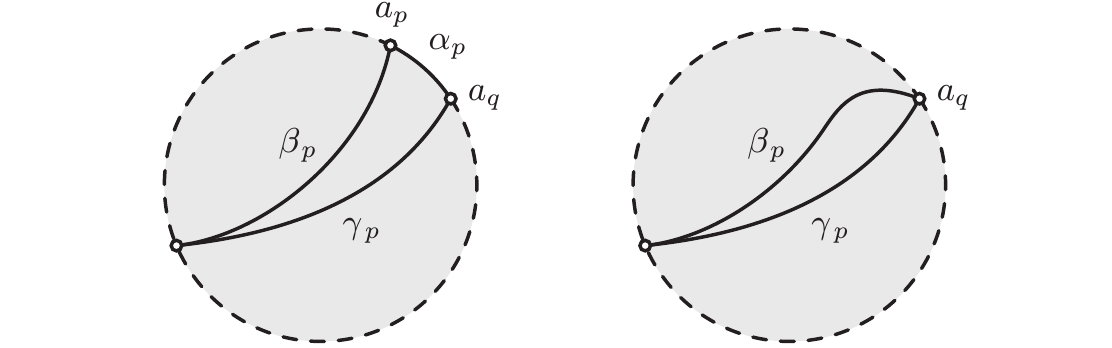}
\caption{The triangle incident to $\alpha_p$ in a triangulation of $\Sigma_n$ (left) and what remains of it when $a_p$ is slid along the boundary to the other vertex $a_q$ of $\alpha_p$ (right).}
\label{Afigure.2.2}
\end{centering}
\end{figure}
Some triangle $t$ of $T$, shown on the left of Fig. \ref{Afigure.2.2}, is incident to arc $\alpha_p$.  As no vertex of $T$ is interior to $\Sigma$ and $n\geq 2$, the two other arcs of $t$ are distinct. Let $\beta_p$ and $\gamma_p$ denote these arcs. Deleting $a_p$ from $T$ consists in displacing $a_p$ along the boundary to the other vertex $a_q$ of $\alpha_p$ within this triangulation, and by removing either $\beta_p$ or $\gamma_p$ from it. In particular, this operation removes vertex $a_p$ and arc $\alpha_p$ from the boundary of $\Sigma$. Since displacing $a_p$ to $a_q$ makes $\beta_p$ and $\gamma_p$ isotopic, as shown on the right of Fig. \ref{Afigure.2.2}, the removal of one of them results in a triangulation of $\Sigma_{n-1}$.

We now describe how modular flip-graphs are affected by vertex deletions.

First observe that the deletion of $a_p$ sends all triangles of $T$ that are not $t$ to triangles. Hence, it transforms a triangulation $T$ that belongs to $\MFS$ into a triangulation in $\MF(\Sigma_{n-1})$ (which appropriately has one less triangle). The resulting triangulation will be denoted by $T{\contract}p$ (following the notation introduced in \cite{Pournin2014}).

The way vertex deletions modify the edges of $\MFS$ is described in \cite{Pournin2014} when $\Sigma$ is a disk and in \cite{ParlierPournin2014} in the more general case of filling surfaces. For completeness, we quickly describe the process here, focussing only on what will be strictly needed in the sequel. Let us consider a path of $\MFS$ between two triangulations $U$ and $V$, i.e., a sequence $(T_i)_{0\leq{i}\leq{k}}$ of triangulations in $\MFS$ where $T_0=U$, $T_k=V$, and $T_{i-1}$ can be transformed into $T_i$ by a flip for $0<i\leq{k}$. This sequence of triangulations can be thought of as a sequence of $k$ flips that transform $U$ into $V$, and for this reason, we refer to $k$ as its length. A path between $U$ and $V$ is called a \emph{geodesic} if its length is equal to the distance between $U$ and $V$ in $\MFS$, which we denote by $d(U,V)$.

Now, call a flip between $U$ and $V$ \emph{incident to arc $\alpha_p$} when $U{\contract}p$ and $V{\contract}p$ are identical. The following theorem is borrowed from \cite{ParlierPournin2014}.

\begin{theorem}\label{Atheorem.2.1}
Let $n$ be an integer greater than $1$. Further consider two triangulations $U$ and $V$ in $\MF(\Sigma_n)$. If there exists a geodesic between $U$ and $V$ along which at least $f$ flips are incident to arc $\alpha_p$, then the following inequality holds:
$$
d(U,V)\geq{d(U{\contract}p,V{\contract}p)+f}\mbox{.}
$$
\end{theorem}

This theorem allows to compare distances within $\MFS$ with distances within $\MF(\Sigma_{n-1})$. It also provides lower bounds on the former that depend on the number of flips incident to a boundary arc $\alpha_p$ along some geodesic between $U$ and $V$. It is not too difficult to find configurations such that this number is positive: for instance when the triangles incident to $\alpha_p$ in $U$ and in $V$ are distinct. In some situations, this number is greater than $1$. Consider, for instance, the triangulations $U$ and $V$ sketched in Fig. \ref{Afigure.2.3}. One can see that the two boundary arcs $\alpha_p$ and $\alpha_q$ that share vertex $a_q$ are incident to the same triangle of $U$. In this case, we say that $U$ has an \emph{ear} in vertex $a_q$. The triangles of $V$ incident to $\alpha_p$ and $\alpha_q$ are distinct and they do not share an arc.
\begin{figure}
\begin{centering}
\includegraphics{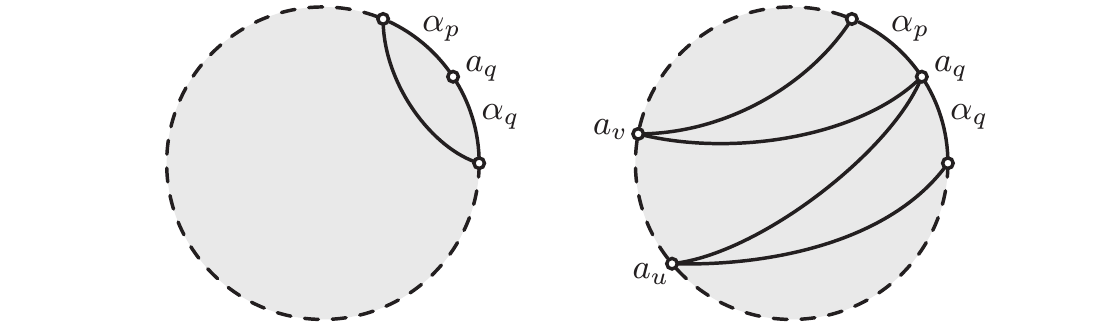}
\caption{Triangulations $U$ (left) and $V$ (right) from Lemma \ref{Alemma.2.1}.}
\label{Afigure.2.3}
\end{centering}
\end{figure}
(Note, however, that $u$ and $v$ are possibly equal.) In this situation, we have the following result, also borrowed from \cite{ParlierPournin2014}.

\begin{lemma}\label{Alemma.2.1}
Consider two triangulations $U$ and $V$ in $\MF(\Sigma_n)$. Further consider two distinct arcs $\alpha_p$ and $\alpha_q$ on the privileged boundary of $\Sigma_n$ so that $a_q$ is a vertex of $\alpha_p$. If $U$ has an ear in $a_q$ and if the triangles of $V$ incident to $\alpha_p$ and to $\alpha_q$ do not have a common arc, then for any geodesic between $U$ and $V$, there exists $r\in\{p,q\}$ so that at least two flips along this geodesic are incident to $\alpha_r$.
\end{lemma}

\section{Upper bounds}\label{sec:upper}

We begin by proving our upper bounds on the diameter of $\mathcal{MF}(\Sigma_n)$ where $\Sigma$ is a one holed torus.

\begin{theorem}
Consider a one holed surface $\Sigma$ of genus $1$ with no interior marked points. The following inequality holds:
$$
\diam(\mathcal{MF}(\Sigma_n))\leq \frac{23}{8} n + 8\mbox{.}
$$
\end{theorem}

\begin{proof}
Let $T$ be a triangulation in $\MF(\Sigma_n)$. Call an arc of $T$ {\it parallel to the  boundary of $\Sigma$} if it is isotopic to an arc that lies within this boundary. Observe that $T$ contains arcs that are not parallel to the boundary, otherwise $\Sigma$ would be a planar surface.

Let $\beta$ be an arc that is not parallel to the boundary of $\Sigma$ and let us consider $\Sigma_n\setminus\beta$. The resulting surface if of genus $0$ with two boundaries, both of which contain a copy of $\beta$ (see Fig. \ref{Afigure.3.1}). 
\begin{figure}
\begin{centering}
\includegraphics{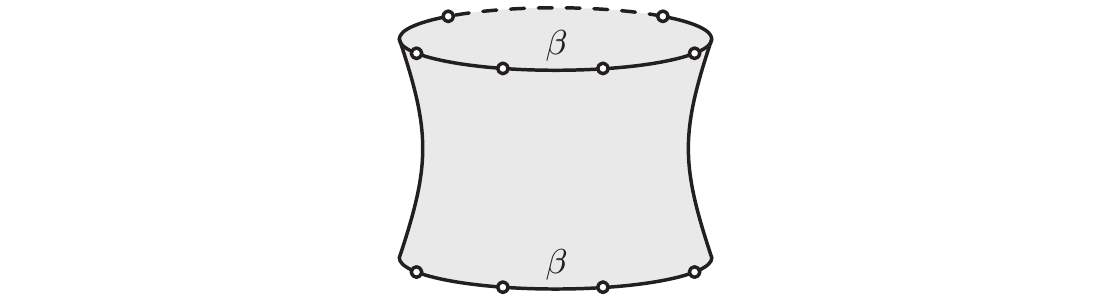}
\caption{Surface $\Sigma_n\setminus\beta$.}\label{Afigure.3.1}
\end{centering}
\end{figure}

The triangulation of $\Sigma_n\setminus\beta$ induced by $T$ must contain an arc between the two boundaries. We choose one and call it $\gamma$. Note that the corresponding arc of $T$ is non-parallel to the boundary of $\Sigma$. Further observe that, by construction, $\beta$ and $\gamma$ are non-homotopic relative to boundary. The surface obtained by cutting $\Sigma_n$ along arcs $\beta$ and $\gamma$ is a polygon, denoted by $P$ in the following of the proof. As the cuts along $\beta$ and $\gamma$ each produce two copies of these arcs, $P$ has $n+4$ vertices.

Consider the $n$ marked points on the boundary of $\Sigma_n$ with their clockwise cyclic order. The endpoints of $\beta$ split the marked points into two sets (we consider the endpoints as being part of both). As $\beta$ and $\gamma$ are non-homotopic and non-parallel to boundary, the endpoints of $\gamma$ are not both contained in exactly one of these sets. Note that several or even all of the endpoints of $\beta$ and $\gamma$ could be the same marked point in the boundary of $\Sigma_n$. The fact that together, the two pairs of endpoints split the boundary vertices into four (possibly empty) sets will be key in the argument. 

Now we consider another triangulation $T'$ in $\MF(\Sigma_n)$ and two arcs $\beta'$ and $\gamma'$ of $T'$, exhibited as $\beta$ and $\gamma$ were in the case of $T$. We denote $P'$ the polygon obtained by cutting $\Sigma_n$ along $\beta'$ and $\gamma'$. There is a sequence of $\lceil{n/8}\rceil$ consecutive marked points on the boundary of $\Sigma_n$ that does not contain any endpoints of $ \beta$, $\gamma$, $\beta'$, or $\gamma'$ except possibly for the first and last vertices in the sequence. The worst case scenario is when the 8 endpoints are evenly spaced around the boundary. We fix a boundary vertex $a_0$ lying in this sequence. 

Let us flip the arcs of the triangulation of $P$ (resp. of $P'$) induced by $T$ (resp. by $T'$) in order to increase the degree of $a_0$ as much as possible. This results in ``fan" like triangulations of $P$ and $P'$ whose interior arcs are all incident to $a_0$. Note that the flips can be chosen so that they each increase the degree of $a_0$ by one. As both polygons have $n+4$ vertices, at most $n+1$ flips are required to do so for each triangulation, and at most $2n + 2$ flips for both of them.

We now flip arcs $\beta$, $\gamma$, $\beta'$, and $\gamma'$ (two of them within each triangulation) to obtain two triangulations in $\MF(\Sigma_n)$ which we denote by $U$ and $U'$. As a result of these flips, we have a pair of loops twice incident to $a_0$ in each of the triangulations $U$ and $U'$. Up to homeomorphism, the pairs of loops, say $(\delta,\epsilon)$ in the first triangulation and $(\delta',\epsilon')$ in the second are the same. We will not flip these arcs anymore. Instead, we will flip the other arcs until the two triangulations coincide; one way of thinking of this is by looking at the triangulations of $\Sigma_n\setminus\{\delta,\epsilon\}$ and $\Sigma_n\setminus\{\delta',\epsilon'\}$ respectively induced by $U$ and $U'$. As polygons $\Sigma_n\setminus\{\delta,\epsilon\}$ and $\Sigma_n\setminus\{\delta',\epsilon'\}$ are identical up to homeomorphism, we can flip arcs of these triangulations until they coincide.

Note that each of the loops that have been created within $U$ and $U'$ has two copies on the boundary of respectively $\Sigma_n\setminus\{\delta,\epsilon\}$ and $\Sigma_n\setminus\{\delta',\epsilon'\}$. Each of these copies is incident to a triangle which we refer to as a \emph{hand}, like the hands of a clock that revolve around vertex $a_0$ (see Fig. \ref{Afigure.3.2}).
\begin{figure}
\begin{centering}
\includegraphics{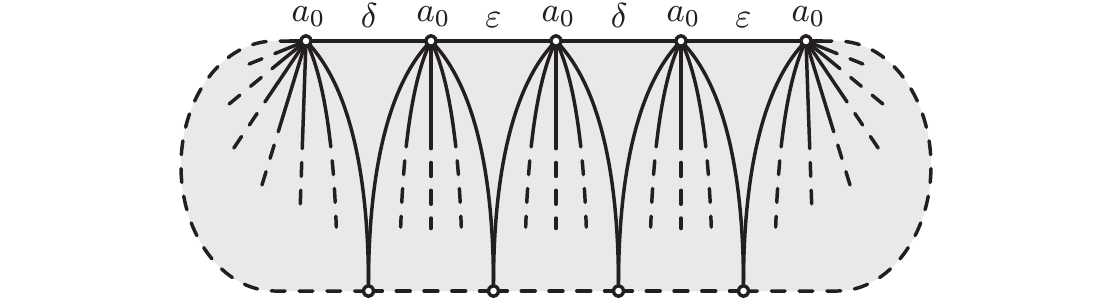}
\caption{The triangulation of $\Sigma_n\setminus\{\delta,\epsilon\}$ induced by $U$.}\label{Afigure.3.2}
\end{centering}
\end{figure}
The unique vertex of each hand that is not a copy of $a_0$ will be referred to as its {\it outer} vertex. Our goal is that both triangulations ``tell the same time": when they do they will be identical up to homeomorphism. 

By construction, the vertex $a_0$ belong to a sequence of at least $\lceil{n/8}\rceil$ consecutive vertices incident to identical arcs in $U$ and in $U'$. We now label clockwise the vertices of polygons $\Sigma_n\setminus\{\delta,\epsilon\}$ and $\Sigma_n\setminus\{\delta',\epsilon'\}$ that are not copies of $a_0$ from $a_1$ to $a_{n-1}$. Consider the smallest $k$ so that $a_k$ is the outer vertex of some hand in either $U$ or $U'$. If $a_k$ is not the outer vertex of a hand in both triangulations, we can move the outer vertex of this hand clockwise using the flip shown on the left of Fig. \ref{Afigure.3.3}. 
\begin{figure}
\begin{centering}
\includegraphics{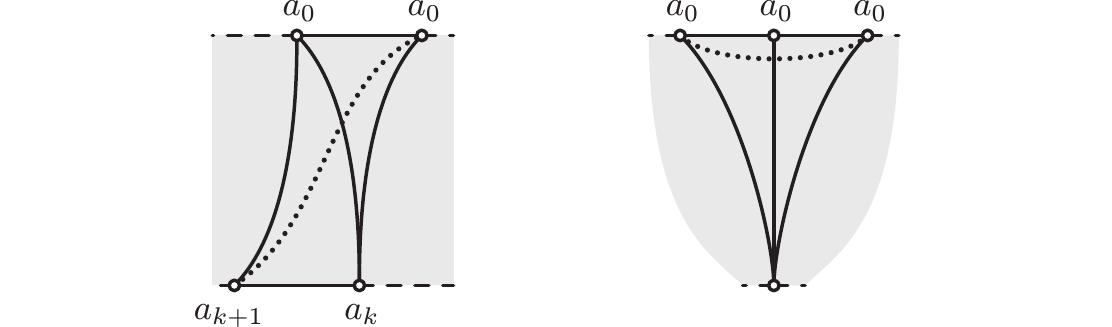}
\caption{A flip that moves the outer vertex of a hand by one vertex clockwise in either $U$ or $U'$ (left) and a flip that binds two hands (right). In both cases, the introduced arc is dotted.}\label{Afigure.3.3}
\end{centering}
\end{figure}

We move the hand until either it coincides with a hand in the other triangulation or it is adjacent to another hand in the same triangulation. In the first case we leave both hands as they are; leaving from their common outer vertex, we search clockwise for the next outer vertex on either triangulation and repeat the process. If two hands are adjacent, we flip the arc they share as shown on the right of Fig. \ref{Afigure.3.3}. This flip introduces an arc twice incident to $a_0$ and allows us to move both hands together. We do so until we either reach the outer vertex of a hand of the other triangulation or if again we meet a third hand on the same triangulation. In the first case we flip the arc twice incident to $a_0$ to unbind the hands and continue with the other hand. In the second case, we introduce an additional arc twice incident to $a_0$ by a flip, in order to move all three hands together. Again we can repeat the process described above until all four hands have the same outer vertices and the triangulations are the same. 

We now count the number of flips that have been necessary to do the whole transformation. We moved the hands around a portion of the boundary spanned by at most $n - \lceil{n/8}\rceil$ vertices. In addition, we may have joined or separated hands several times; because we have 4 hands in total, we may have had to perform at most 6 binding or unbinding flips. In total the number of flips did not exceed 
$$
\frac{7}{8} n + 6\mbox{.}
$$
Therefore, together with the first step, we needed no more than
$$
\frac{23}{8} n + 8
$$
flips to transform $T$ into $T'$, as claimed.
\end{proof}

Using the same ideas, one can prove upper bounds for higher genus $\Sigma$. According to these, the growth rate of $\diam(\mathcal{MF}(\Sigma_n))/n$ is less than $4$ for any genus, improving in this  case on the general upper bound of $4$ on the growth rate, shown in \cite{ParlierPournin2014} for any filling surface. The methods are almost identical to the above proof, so we only provide a sketch of the proof noting the differences along the way.

\begin{theorem}
Let  $\Sigma$ be a one holed surface of genus $g\geq2$ with no interior marked points. The following inequality holds:
$$
\diam(\mathcal{MF}(\Sigma_n))\leq \left(4 - \frac{1}{4g}\right)  n + K_g\mbox{,}
$$
where $K_g$ is a constant that only depends on $g$. 
\end{theorem}

\begin{proof}
Within any triangulation $T$ in $\MF(\Sigma_n)$, there are $2g$ arcs $\beta_1,\hdots,\beta_{2g}$ such that $\Sigma_n\setminus\{\beta_1,\hdots,\beta_{2g}\}$ is a polygon. Now given two triangulations, $T$ and $T'$, we choose such sets of arcs. Cutting along these arcs gives us two triangulated polygons $P$ and $P'$. We can find a sequence of $n/(8g)$ consecutive vertices on the boundary of $\Sigma_n$ that do not contain any of the endpoints of the $2g$ arcs on either triangulation. We will call this sequence of vertices the {\it untouched} vertices. We choose a vertex $a_0$ in that sequence, and we flip the interior arcs of the triangulations of polygons $P$ and $P'$ induced by $T$ and $T'$ in order to increase the degree of $a_0$ as much as possible. We further flip the $2g$ arcs along which the cut was performed within each triangulation to get $2g$ loops twice incident to $a_0$ which generate the topology of $\Sigma_n$. The total number of flips that have been performed so far is at most 
$$
2(n + 4g - 3 + 2g)= 2n +12g -6\mbox{.}
$$

We can now cut $\Sigma_n$ along the $2g$ loops within each triangulation to get two new polygons with $n+4g$ vertices. Each polygon has $4g$ particular boundary arcs obtained as copies of the loops and, therefore, whose endpoints are copies of $a_0$. The remaining $n$ boundary arcs of the polygons are the boundary arcs of $\Sigma_n$. 

Again we have a notion of hands (this time we have $4g$ of them) corresponding to triangles incident to the particular arcs of the polygons. Here we note an important difference with what happened for the torus: in the case of the torus, there was only one type of topological configuration of the loops in $\Sigma_n$ and there was a self homeomorphism of $\Sigma_n$ between the sets of $2g$ loops. In general however, this fails to be the case (for instance, one can think of the different ways an octagon can be glued along its boundary arcs in order to get a surface of genus $2$). So instead of just requiring that the hands have the same outer vertices, we will perform flips until they all have the same outer vertex. We leave from the last untouched vertex clockwise and, using flips, we move the hands clockwise to the first untouched vertex. Within each triangulation, this requires 
$$
n\left(1 - \frac{1}{8g}\right)
$$
moving flips and in addition $4g -1$ binding flips to put the hands together. The two triangulations now only differ on a subsurface of genus $g$ with one boundary and a single marked point on it. They can be made identical by flipping inside this subsurface; this requires at most $D_g$ flips where $D_g$ is the flip graph diameter of this subsurface (and can be made explicit, see \cite{DisarloParlier2014}).

The total number of flips performed is
$$
2n +12g -6 + 2n \left( 1 - \frac{1}{8g}\right) + 4g -1 + D_g\mbox{,}
$$
which is precisely equal to
$$
\left(4 -  \frac{1}{4g}\right) n + K_g\mbox{,}
$$
where $K_g= 16 g - 7 + D_g$. This is the result claimed. 
\end{proof}
\section{Lower bounds}\label{Asection.4}

Let $\Sigma$ be a one holed surface with strictly positive genus $g$ and no interior marked points. Using the techniques from \cite{ParlierPournin2014,Pournin2014}, we prove in this section that
$$
\diam(\mathcal{MF}(\Sigma_n))\geq\left\lfloor{\frac{5}{2}n}\right\rfloor-2\mbox{.}
$$

As in \cite{ParlierPournin2014,Pournin2014} we exhibit two triangulations $A_n^-$ and $A_n^+$ in $\MF(\Sigma_n)$ whose distance is at least the desired lower bound. These triangulations are built from the triangulation $Z_n$ of $\Delta_n$ depicted in Fig. \ref{Afigure.4.1}, where $\Delta_n$ is a disk with $n$ marked vertices on the boundary. The interior arcs of $Z_n$ form a zigzag, i.e., a simple path that alternates between left and right turns.
\begin{figure}
\begin{centering}
\includegraphics{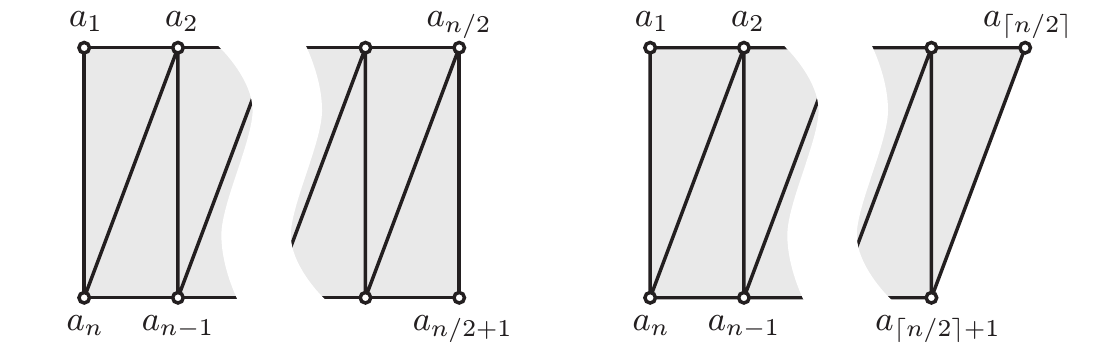}
\caption{The triangulation $Z_n$ of $\Delta_n$ depicted when $n$ is even (left) and odd (right).}\label{Afigure.4.1}
\end{centering}
\end{figure}
This path starts at vertex $a_n$, ends at vertex $a_{n/2}$ when $n$ is even, and at vertex $a_{\lceil{n/2\rceil+1}}$ when $n$ is odd. When $n$ is greater than $3$, triangulation $Z_n$ has an ear in $a_1$ and another ear in $a_{\lfloor{n/2}\rfloor+1}$. When $n$ is equal to $3$, this triangulation is made up of a single triangle which is an ear in all three vertices. Observe that $Z_n$ cannot be defined when $n$ is less than $3$.

Assume that $n\geq3$. A triangulation $A_n^-$ of $\Sigma_n$ can be built by re-triangulating the ear of $Z_n$ in $a_1$ as follows: first place an arc with vertices $a_1$ and $a_n$ within the ear. This arc is isotopic to $\alpha_n$. Next, place a loop $\gamma$ twice incident to $a_1$ between the two isotopic arcs, remove the disk enclosed by this loop, and replace it by a triangulation $A_1^-$ of $\Sigma_1$. The resulting triangulation of $\Sigma_n$ is shown at the top of Fig. \ref{Afigure.4.2}. In this representation, the triangle $t$ of $A_1^-$ incident to $\gamma$ is depicted, but the rest of $A_1^-$ is omitted, as indicated by the hatched surface. Observe that, since $a_1$ is the only vertex of $A_1^-$, then $t$ must be bounded by $\gamma$ and two other loops twice incident to $a_1$ as shown on the figure.

Another triangulation $A_n^+$ of $\Gamma_n$ can be built by re-triangulating the ear of $Z_n$ in $a_{\lfloor{n/2}\rfloor+1}$ in the same way, except that $\gamma$ is chosen incident to $a_{\lfloor{n/2}\rfloor+1}$ instead of $a_1$. Also, a triangulation $A_1^+$ of $\Sigma_1$, possibly distinct from $A_1^-$, replaces the disk enclosed by $\gamma$. Triangulation $A_n^+$ is shown in the bottom of Fig. \ref{Afigure.4.2}.

\begin{figure}
\begin{centering}
\includegraphics{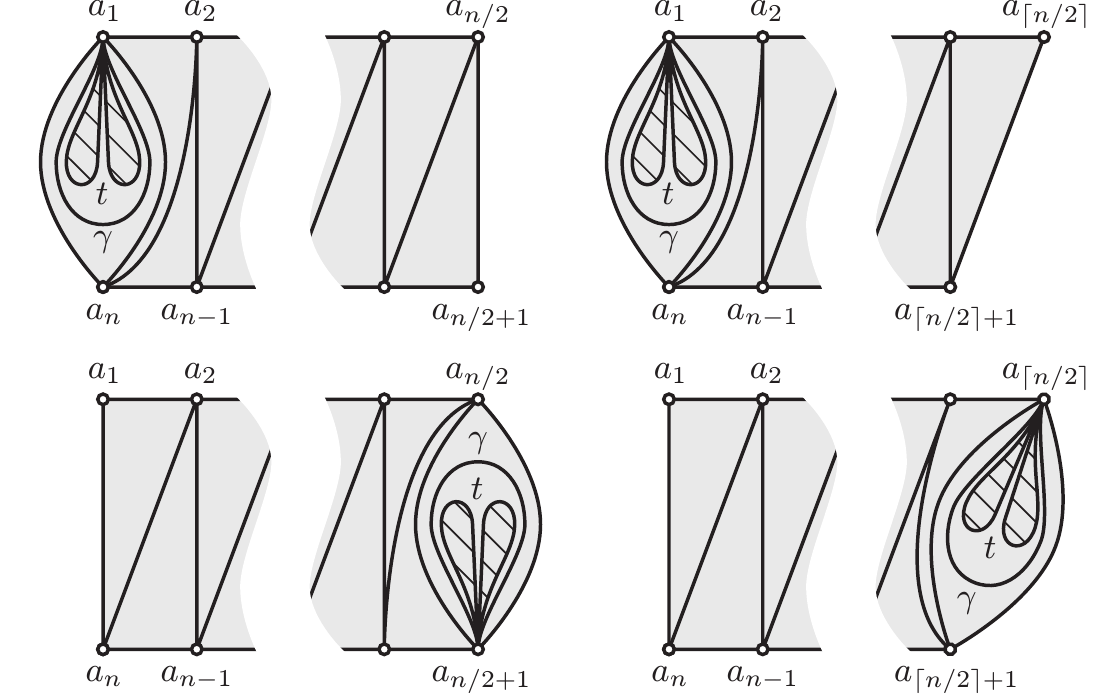}
\caption{The triangulations $A_n^-$ (top row) and $A_n^+$ (bottom row) of $\Sigma_n$ depicted when $n$ is even (left) and odd (right). For simplicity, the triangulations of $\Sigma_1$ used to build $A_n^-$ and $A_n^+$ are mostly omitted here (hatched surface), except for two interior loops twice incident to $a_1$ (top) and to $a_{\lfloor{n/2}\rfloor+1}$ (bottom).}\label{Afigure.4.2}
\end{centering}
\end{figure}

In this construction, $A_1^-$ and $A_1^+$ are arbitrary triangulations in $\MF(\Sigma_1)$. We can also define $A_2^-$ and $A_2^+$ as the triangulations in $\MF(\Sigma_2)$ that contain a loop arc at respectively vertex $a_1$ and vertex $a_2$, in such a way that these arcs respectively enclose triangulations $A_1^-$ and $A_1^+$.

In order to obtain the desired lower bound, one of the main steps will be to show the following inequality for every integer $n$ greater than $2$:
\begin{equation}\label{Aequation.4.1}
d(A_n^-,A_n^+)\geq\min(\{d(A_{n-1}^-,A_{n-1}^+)+3,d(A_{n-2}^-,A_{n-2}^+)+5\})\mbox{.}
\end{equation}

This inequality will be obtained using well chosen vertex deletions. Observe that for $n\geq2$, deleting vertex $a_n$ from both $A_n^-$ and $A_n^+$ results in triangulations isomorphic to $A_{n-1}^-$ and $A_{n-1}^+$. More precisely, once $a_n$ has been deleted, the other vertices need be relabeled in order to obtain $A_{n-1}^-$ and $A_{n-1}^+$. If the deleted vertex is $a_j$, the natural way to do so is to relabel by $a_{i-1}$ every vertex $a_i$ so that $i>j$.

This relabeling provides a map to the triangulations of $\Gamma_{n-1}$. In the following, we call any such map a {\it vertex relabeling}. The above observation can be reformulated as follows: $A_{n}^-\contract{n}$, resp. $A_{n}^+\contract{n}$ are isomorphic to $A_{n-1}^-$, resp. $A_{n-1}^+$ via the same vertex relabeling. This can be checked using Fig. \ref{Afigure.4.2}.

It then follows from Theorem \ref{Atheorem.2.1} that, if there exists a geodesic between $A_n^-$ and $A_n^+$ with at least $3$ flips incident to $\alpha_n$, then
\begin{equation}\label{Aequation.4.2}
d(A_n^-,A_n^+)\geq{d(A_{n-1}^-,A_{n-1}^+)+3}\mbox{,}
\end{equation}
and inequality (\ref{Aequation.4.1}) holds in this case. Now assume that $n\geq3$ and observe that for any integer $i$ so that $1\leq{i}<n$ and any $j\in\{n-i,n-i+1\}$, deleting vertices $a_i$ and $a_j$ from $A_n^-$ and from $A_n^+$ results in triangulations of $\Gamma_n$ isomorphic to $A_{n-2}^-$ and $A_{n-2}^+$ respectively by the same vertex relabeling. Hence, if there exists a geodesic path between $A_n^-$ and $A_n^+$ with at least $3$ flips incident to $\alpha_i$, and a geodesic path between $A_n^-\contract{i}$ and $A_n^+\contract{i}$ with at least $2$ flips incident to $\alpha_j$ then, invoking Theorem \ref{Atheorem.2.1} twice yields
\begin{equation}\label{Aequation.4.3}
d(A_n^-,A_n^+)\geq{d(A_{n-2}^-,A_{n-2}^+)+5}\mbox{,}
\end{equation}
and (\ref{Aequation.4.1}) also holds in this case. Observe that (\ref{Aequation.4.2}) and (\ref{Aequation.4.3}) follow from the existence of particular geodesic paths. The rest of the section is devoted to proving the existence of geodesic paths that imply at least one of these inequalities.

Observe that $\alpha_n$ is not incident to the same triangle in $A_n^-$ and in $A_n^+$. Therefore, at least one flip is incident to $\alpha_n$ along any geodesic from $A_n^-$ to $A_n^+$. We will study these geodesics depending on the arc introduced by their first flip incident to $\alpha_n$; this is the purpose of the next three lemmas.

\begin{lemma}\label{Alemma.4.1}
For some integer $n$ greater than $2$, consider a geodesic from $A_n^-$ to $A_n^+$. If the first flip incident to arc $\alpha_n$ along this geodesic introduces an arc with vertices $a_1$ and $a_n$, then there must be at least two other flips incident to $\alpha_n$ along it.
\end{lemma}
\begin{proof}
Let $(T_i)_{0\leq{i}\leq{k}}$ be a geodesic from $A_n^-$ to $A_n^+$. Assume that the first flip incident to $\alpha_n$ along $(T_i)_{0\leq{i}\leq{k}}$ is the $j$-th one, and that it introduces an arc $\beta$ with vertices $a_1$ and $a_n$.
\begin{figure}
\begin{centering}
\includegraphics{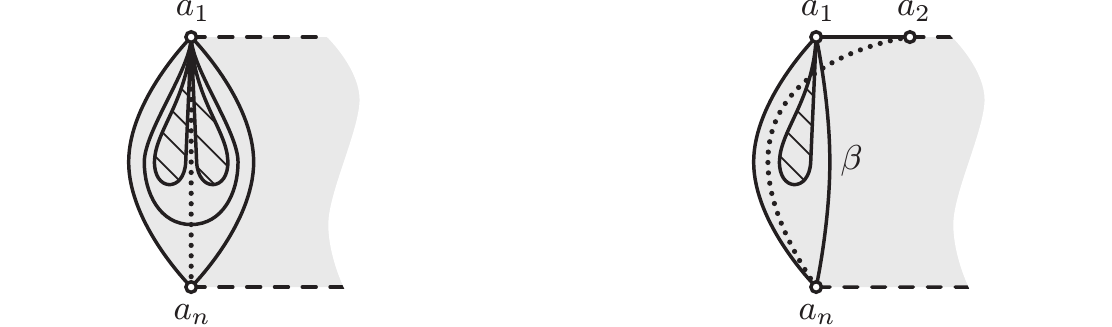}
\caption{The $j$-th (left) and $l$-th (right) flips performed along path $(T_i)_{0\leq{i}\leq{k}}$ in the proof of Lemma \ref{Alemma.4.1}. In each case, the introduced arc is dotted and the solid arcs belong to triangulation $T_{i-1}$.}\label{Afigure.4.3}
\end{centering}
\end{figure}
This flip must then be the one shown on the left of Fig. \ref{Afigure.4.3}.

As the triangles incident to $\alpha_n$ in $T_j$ and in $A_n^+$ are distinct, at least one of the last $k-j$ flips along our geodesic is incident to this arc. Assume that there is exactly one such flip, say the $l$-th one. This flip then necessarily introduces an ear in vertex $a_1$. Hence, the arc introduced by it must be the one shown as a dotted line on the right of Fig. \ref{Afigure.4.3}. This is impossible because this arc crosses two interior arcs of $T_{l-1}$: arc $\beta$, and the arc twice incident to $a_1$ that bounds the triangle incident to $\alpha_n$ in $T_{l-1}$. Therefore, at least three flips are incident to $\alpha_n$ along $(T_i)_{0\leq{i}\leq{k}}$.
\end{proof}

The proof of the following lemma is, amost word for word, that of Lemma 4.2 in \cite{ParlierPournin2014}. We include it for completeness.

\begin{lemma}\label{Alemma.4.2}
Let $n$ be an integer greater than $2$. Consider a geodesic from $A_n^-$ to $A_n^+$ whose first flip incident to $\alpha_n$ introduces an arc with vertices $a_1$ and $a_2$. If $\alpha_n$ is incident to at most $2$ flips along this geodesic then $\alpha_1$ is incident to at least $4$ flips along it.
\end{lemma}
\begin{proof}
Let $(T_i)_{0\leq{i}\leq{k}}$ be a geodesic from $A_n^-$ to $A_n^+$ whose first flip incident to $\alpha_n$, say the $j$-th one, introduces an arc with vertices $a_1$ and $a_2$. This flip must then be the one shown on the left of Fig. \ref{Afigure.4.4}. Note that it is incident to arc $\alpha_1$.

Assume that at most one flip along $(T_i)_{0\leq{i}\leq{k}}$ other than the $j$-th one is incident to $\alpha_n$. In this case, there must be exactly one such flip among the last $k-j$ flips of $(T_i)_{0\leq{i}\leq{k}}$, say the $l$-th one. Moreover, this flip replaces the triangle of $T_j$ incident to $\alpha_n$ by the triangle of $A_n^+$ incident to $\alpha_n$. There is only one way to do so, as depicted on the right of Fig. \ref{Afigure.4.4}. Note that this flip is also incident to $\alpha_1$.

Finally, as the arc introduced by the $j$-th flip along $(T_i)_{0\leq{i}\leq{k}}$ is not removed before the $l$-th flip, there must be two more flips incident to arc $\alpha_1$ along this geodesic: the flip that removes the loop arc with vertex $a_1$ shown on the left of Fig. \ref{Afigure.4.4} and the flip that introduces the loop arc with vertex $a_2$ shown on the right of the figure. This proves that at least four flips are incident to $\alpha_1$ along $(T_i)_{0\leq{i}\leq{k}}$.
\end{proof}

\begin{figure}
\begin{centering}
\includegraphics{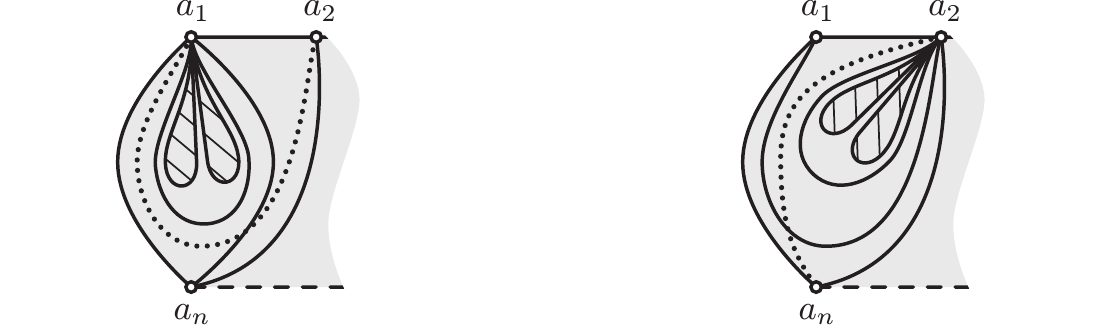}
\caption{The $j$-th (left) and $l$-th (right) flip performed along path $(T_i)_{0\leq{i}\leq{k}}$ in the proof of Lemma \ref{Alemma.4.2}. In each case, the introduced arc is dotted and the solid arcs belong to triangulation $T_{i-1}$.}\label{Afigure.4.4}
\end{centering}
\end{figure}

The proof of the last of the three lemmas is slightly more involved.

\begin{lemma}\label{Alemma.4.3}
For $n\geq4$, consider a geodesic from $A_n^-$ to $A_n^+$ whose first flip incident to $\alpha_n$ introduces an arc with vertices $a_1$ and $a_p$, where $2<p<n$. Then:
$$
d(A_n^-,A_n^+)\geq\min(\{d(A_{n-1}^-,A_{n-1}^+)+3,d(A_{n-2}^-,A_{n-2}^+)+5\})\mbox{.}
$$
\end{lemma}
\begin{proof}
Let $(T_i)_{0\leq{i}\leq{k}}$ be a geodesic from $A_n^-$ to $A_n^+$ whose first flip incident to $\alpha_n$, say the $j$-th one, introduces an arc with vertices $a_1$ and $a_p$, where $2<p<n$. This flip is depicted in Fig. \ref{Afigure.4.5}. As a first step, we will show that $T_j$ has an ear in some vertex $a_q$ where either $2\leq{q}<\lceil{n/2}\rceil$ or $\lceil{n/2}\rceil+1<q\leq{n}$.

Assume that $p$ is not greater than $\lceil{n/2}\rceil$. Consider the arc of $T_j$ with vertices $a_1$ and $a_p$ shown as a solid line on the left of Fig. \ref{Afigure.4.5}.
\begin{figure}
\begin{centering}
\includegraphics{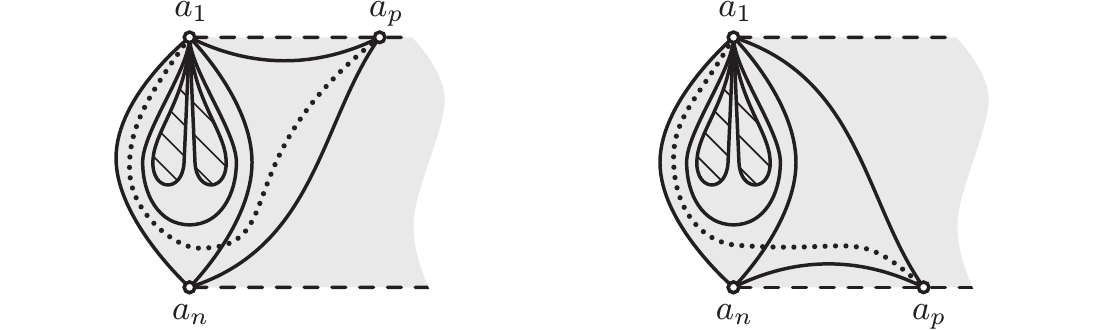}
\caption{The $j$-th flip performed along path $(T_i)_{0\leq{i}\leq{k}}$ in the proof of Lemma \ref{Alemma.4.3} when $p\leq{\lceil{n/2}\rceil}$ (left) and when $p>\lceil{n/2}\rceil$ (right). In each case, the introduced arc is dotted and the solid edges belong to triangulation $T_{j-1}$.}\label{Afigure.4.5}
\end{centering}
\end{figure}
The portion of $T_j$ bounded by this arc and by arcs $\alpha_1$, ..., $\alpha_{p-1}$ is a triangulation of disk $\Delta_p$. If $p>3$, then this triangulation has at least two ears, and one of them is also an ear of $T_j$ in some vertex $a_q$ where $2\leq{q}<p$. As $p\leq\lceil{n/2}\rceil$, the desired inequalities on $q$ hold. If $p=3$, then $T_j$ has an ear in $a_2$ since the triangulation of $\Delta_p$ induced by $T_j$ is made up of a single triangle, and these inequalities on $q$ also hold. 

Now assume that $\lceil{n/2}\rceil<p<n$. Consider the arc with vertices $a_1$ and $a_p$ introduced by the $j$-th flip along $(T_i)_{0\leq{i}\leq{k}}$ and shown as a dotted line on the right of Fig. \ref{Afigure.4.5}. The portion of $T_j$ bounded by this arc and by arcs $\alpha_p$, ..., $\alpha_n$ is a triangulation of $\Delta_{n-p+2}$. Since $n-p+2$ is at least $3$, an argument similar to the one used in the last paragraph shows that $T_j$ has an ear in some vertex $a_q$ where $p<q\leq{n}$. As $p$ is greater than $\lceil{n/2}\rceil$, we obtain $\lceil{n/2}\rceil+1<q\leq{n}$, as desired.

Note that, as $2\leq{q}<\lceil{n/2}\rceil$ or $\lceil{n/2}\rceil+1<q\leq{n}$, the triangles incident to arcs $\alpha_{q-1}$ and to $\alpha_q$ in either $A_n^-$ or $A_n^+$ do not share an arc. Since $T_j$ has an ear in $a_q$, Lemma \ref{Alemma.2.1} can be invoked with $U=T_j$ for the two geodesics obtained by splitting $(T_i)_{0\leq{i}\leq{k}}$ at $T_j$. Doing so, we find that either $\alpha_{q-1}$ and $\alpha_q$ are both incident to exactly $3$ flips along this geodesic or one of these arcs is incident to at least $4$ flips along it. We will review the two cases separately.

If $\alpha_r$ is incident to at least $4$ flips along $(T_i)_{0\leq{i}\leq{k}}$, where $r$ is equal to $q-1$ or to $q$, then Theorem \ref{Atheorem.2.1} yields
\begin{equation}\label{Alemma.4.3.equation.1}
d(A_n^-,A_n^+)\geq{d(A_n^-\contract{r},A_n^+\contract{r})+4}\mbox{.}
\end{equation}

Call $s=n-q+1$ and observe that arc $\alpha_s$ is not incident to the same triangle in $A_n^-\contract{r}$ and in $A_n^+\contract{r}$. Hence, some flip must be incident to this arc along any geodesic between $A_n^-\contract{r}$ and $A_n^+\contract{r}$. Invoking Theorem \ref{Atheorem.2.1} again, we find
\begin{equation}\label{Alemma.4.3.equation.2}
d(A_n^-\contract{r},A_n^+\contract{r})\geq{d(A_n^-\contract{r}\contract{s},A_n^+\contract{r}\contract{s})+1}\mbox{.}
\end{equation}

As $A_n^-\contract{r}\contract{s}$ and $A_n^+\contract{r}\contract{s}$ are isomorphic to $A_{n-2}^-$ and $A_{n-2}^+$ by the same vertex relabeling, the desired result is obtained combining (\ref{Alemma.4.3.equation.1}) and (\ref{Alemma.4.3.equation.2}).

Now assume that $\alpha_{q-1}$ and $\alpha_q$ are both incident to exactly $3$ flips along $(T_i)_{0\leq{i}\leq{k}}$. If $q=n$, then the result immediately follows from Theorem \ref{Atheorem.2.1} because $A_n^-\contract{n}$ and $A_n^+\contract{n}$ are isomorphic to $A_{n-1}^-$ and $A_{n-1}^+$ by the same vertex relabeling. We will therefore also assume that $q<n$. 

Note that at least one of the first $j$ flips and at least one of the last $k-j$ flips along $(T_i)_{0\leq{i}\leq{k}}$ are incident to either $\alpha_{p-1}$ or $\alpha_p$ because these arcs are not incident to the same triangles in $T_j$ and in either $A_n^-$ or $A_n^+$.
\begin{figure}[b]
\begin{centering}
\includegraphics{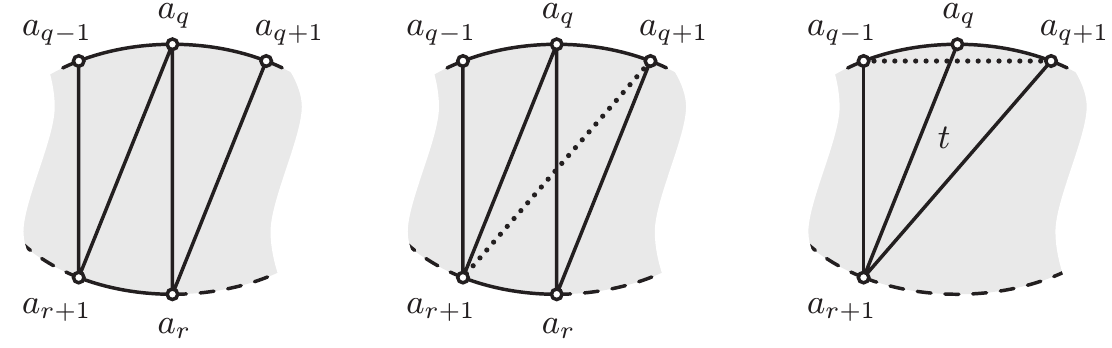}
\caption{The portion of either $A_n^-$ or $A_n^+$ next to vertex $a_q$ (left), the $l$-th flip along the geodesic used in the proof of Lemma \ref{Alemma.4.3} (center), and the first flip incident to $\alpha_{q-1}$ along this geodesic (right). The arc introduced by each flip is dotted.}\label{Afigure.4.6}
\end{centering}
\end{figure}
We can in fact assume without loss of generality that exactly one of the first $j$ flips and two of the last $k-j$ flips along $(T_i)_{0\leq{i}\leq{k}}$ are incident to $\alpha_{q-1}$ by, if needed, reversing geodesic $(T_i)_{0\leq{i}\leq{k}}$ and exchanging $A_n^-$ with $A_n^+$. In this case, exactly two of the first $j$ flips and exactly one of the last $k-j$ flips along this geodesic are incident to $\alpha_q$.

Call $r=n-q+1$. The portion of either $A_n^-$ or $A_n^+$ placed between arcs $\alpha_{q-1}$, $\alpha_q$, and $\alpha_r$ is depicted on the left of Fig. \ref{Afigure.4.6}.

Since only one flip is incident to $\alpha_{q-1}$ along $(T_i)_{0\leq{i}\leq{j}}$, this flip must be the one shown on the right of Fig. \ref{Afigure.4.6}, and it is also incident to $\alpha_{q}$. Now consider the triangle incident to $\alpha_q$ when this flip is performed, labeled $t$ in the figure. This triangle must be introduced by the first flip incident to $\alpha_q$, earlier along the geodesic. This flip, say the $l$-th one, is shown in the center of Fig. \ref{Afigure.4.6}.

Consider a geodesic $(T'_i)_{0\leq{i}\leq{k'}}$ from $A_n^-\contract{q-1}$ to $T_l\contract{q-1}$, and a geodesic $(T''_i)_{j\leq{i}\leq{k''}}$ from $T_l\contract{q-1}$ to $A_n^+\contract{q-1}$. Since three flips are incident to $\alpha_{q-1}$ along $(T_i)_{0\leq{i}\leq{k}}$, it follows from Theorem \ref{Atheorem.2.1} that
\begin{equation}\label{Alemma.4.3.equation.3}
k'+k''\leq{d(A_n^-,A_n^+)-3}\mbox{.}
\end{equation}

Observe that the triangles incident to $\alpha_r$ in $A_n^-\contract{q-1}$ and in $T_l\contract{q-1}$ are distinct. Hence, at least one flip is incident to $\alpha_r$ along $(T'_i)_{0\leq{i}\leq{k'}}$ and by Theorem \ref{Atheorem.2.1},
\begin{equation}\label{Alemma.4.3.equation.4}
k'\geq{d(A_n^-\contract{q-1}\contract{r},T_l\contract{q-1}\contract{r})+1}\mbox{.}
\end{equation}

Similarly,  the triangles incident to $\alpha_r$ in $T_l^-\contract{q-1}$ and in $A_n^+\contract{q-1}$ are distinct. Hence, at least one flip is incident to $\alpha_r$ along $(T''_i)_{0\leq{i}\leq{k''}}$ and by Theorem \ref{Atheorem.2.1},
\begin{equation}\label{Alemma.4.3.equation.5}
k''\geq{d(T_l\contract{q-1}\contract{r},A_n^+\contract{q-1}\contract{r})+1}\mbox{.}
\end{equation}

By the triangle inequality, (\ref{Alemma.4.3.equation.4}) and (\ref{Alemma.4.3.equation.5}) yield
\begin{equation}\label{Alemma.4.3.equation.6}
k'+k''\geq{d(A_n^-\contract{q-1}\contract{r},A_n^+\contract{q-1}\contract{r})+2}\mbox{.}
\end{equation}

Since $A_n^-\contract{q-1}\contract{r}$ and $B_n^+\contract{q-1}\contract{r}$ are isomorphic to $A_{n-2}^-$ and $A_{n-2}^+$ by the same vertex relabeling, the desired inequality is obtained combining (\ref{Alemma.4.3.equation.3}) and (\ref{Alemma.4.3.equation.6}).
\end{proof}

We can now prove the main inequality.

\begin{theorem}\label{Atheorem.4.1}
For every integer $n$ greater than $2$,
$$
d(A_n^-,A_n^+)\geq\min(\{d(A_{n-1}^-,A_{n-1}^+)+3,d(A_{n-2}^-,A_{n-2}^+)+5\})\mbox{.}
$$
\end{theorem}
\begin{proof}
Assume that $n\geq3$ and consider a geodesic $(T_i)_{0\leq{i}\leq{k}}$ from $A_n^-$ to $A_n^+$. If at least $3$ flips are incident to $\alpha_n$ along it, then Theorem \ref{Atheorem.2.1} yields
$$
d(A_n^-,A_n^+)\geq{d(A_{n-1}^-,A_{n-1}^+)+3}\mbox{.}
$$

Indeed, as mentioned before, $A_n^-\contract{n}$ and $A_n^+\contract{n}$ are respectively isomorphic to $A_{n-1}^-$ and $A_{n-1}^+$ via the same vertex relabeling. Hence the result holds in this case and we assume in the remainder of the proof that at most two flips are incident to $\alpha_n$ along $(T_i)_{0\leq{i}\leq{k}}$. By Lemma \ref{Alemma.4.1}, the first flip incident to $\alpha_n$ along $(T_i)_{0\leq{i}\leq{k}}$, say the $j$-th one, must introduce an arc with vertices $a_1$ and $a_p$ where $2\leq{p}<n$. If $p$ is equal to $2$ then, by Lemma \ref{Alemma.4.2} and Theorem \ref{Atheorem.2.1},
\begin{equation}\label{Aequation.4.4}
d(A_n^-,A_n^+)\geq{d(A_n^-\contract{1},A_n^+\contract{1})+4}\mbox{.}
\end{equation}

Observe that arc $\alpha_{n-1}$ is not incident to the same triangle in $A_n^-\contract{1}$ and in $A_n^+\contract{1}$. Therefore, there must be at least one flip incident to $\alpha_{n-1}$ along any geodesic between these triangulations. Therefore, invoking Theorem \ref{Atheorem.2.1},
\begin{equation}\label{Aequation.4.5}
d(A_n^-\contract{1},A_n^+\contract{1})\geq{d(A_n^-\contract{1}\contract{n-1},A_n^+\contract{1}\contract{n-1})+1}\mbox{.}
\end{equation}

As $A_n^-\contract{1}\contract{n-1}$ and $A_n^+\contract{1}\contract{n-1}$ are isomorphic to $A_{n-2}^-$ and $A_{n-2}^+$ via the same vertex relabeling, the result is obtained combining (\ref{Aequation.4.4}) and (\ref{Aequation.4.5}).

Finally, if the $j$-th flip introduces an arc with vertices $a_1$ and $a_p$, where $2<p<n$, then $n$ must be greater than $3$ and the result follows from Lemma \ref{Alemma.4.3}.
\end{proof}

We conclude the following.

\begin{theorem}\label{Atheorem.4.2}
Let $\Sigma$ be a one holed surface of genus $g>0$ with no interior marked points. The following inequality holds:
$$
\displaystyle\diam(\MFS)\geq\left\lfloor{\frac{5}{2}n}\right\rfloor-2\mbox{.}
$$
\end{theorem}
\begin{proof}
By Theorem \ref{Atheorem.4.1}, the result is easily proven by induction on $n$, provided it holds when $n$ is equal to $1$ or to $2$. In $n$ is equal to $1$, the it is obtained immediately because $\lfloor{5n/2}\rfloor=2$ and $\diam(\MF(\Sigma_1))\geq0$.

Assume that $n=2$.
\begin{figure}
\begin{centering}
\includegraphics{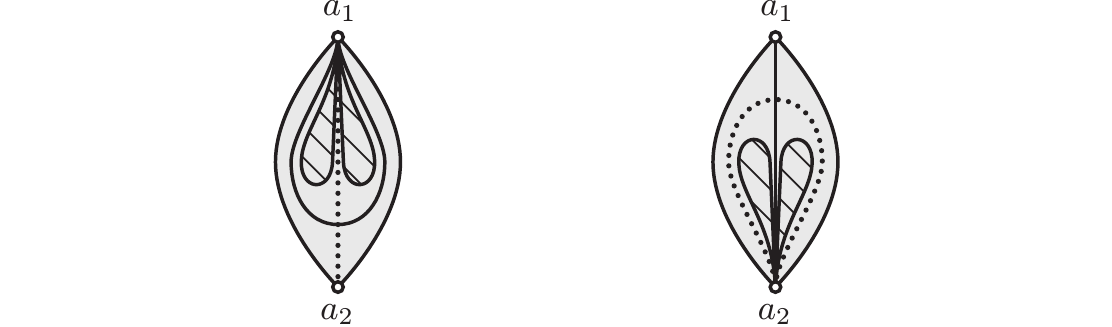}
\caption{The first (left) and last (right) flips incident to $\alpha_2$ along any geodesic from $A_2^-$ to $A_2^+$. In both cases, the introduced arc is dotted.}\label{Afigure.4.7}
\end{centering}
\end{figure}
In this case, we can prove that at least $3$ flips are incident to arc $\alpha_2$ along any geodesic from $A_2^-$ and $A_2^+$. Indeed, consider the triangle incident to $\alpha_2$ in $A_2^-$. The only arc of this triangle that can be flipped is the loop twice incident to $a_1$. As a consequence, the first flip incident to $\alpha_2$ along any geodesic from $A_2^-$ to $A_2^+$ is as shown on the left of Fig. \ref{Afigure.4.7}. Similarly, the last flip incident to $\alpha_2$ along a such geodesic must be the one depicted on the right of the figure. Note that these flips are distinct as they do not remove the same arc. Now consider the triangles incident to $\alpha_2$ in the triangulation obtained from the first of these flips and in the triangulation within which the second is performed. These triangles are distinct because one of them has an arc twice incident to $a_1$ and the other an arc twice incident to $a_2$. Hence, there must be a third flip incident to $\alpha_2$ along any geodesic from $A_2^-$ to $A_2^+$. It therefore follows from Theorem \ref{Atheorem.2.1} that $d(A_2^-,A_2^+)\geq3$, which is precisely the desired inequality when $n=2$.
\end{proof}

The additive constant in the lower bound provided by this theorem can be easily improved to take the genus into account. Indeed, recall that $A_1^-$ and $A_1^+$ have been chosen in an arbitrary way at the beginning of the section. However, as shown in \cite{DisarloParlier2014} (see Corollary 4.19 therein) the diameter of $\MF(\Sigma_1)$ is bounded below by $K\,g\log(g)$, for some positive constant $K$ and large enough $g$. Therefore, triangulations $A_1^-$ and $A_1^+$ can be chosen in such a way that, from the argument used in the proof of Theorem \ref{Atheorem.4.2}, the additive constant in our lower bound is $K\,g\log(g)-2$ instead of $-2$, provided $g$ is large enough.

Note, however, that this additive constant cannot be improved in the case of the one holed torus, i.e., when $g=1$. Indeed, in this case, a triangulation of $\Sigma_1$ is made up of three triangles: one of them is incident to the boundary arc and adjacent to the other two triangles $t^-$ and $t^+$. As $g=1$, the way $t^-$ and $t^+$ are glued to one another is prescribed. It follows that $\Sigma_1$ has exactly one triangulation up to homeomorphism and that the diameter of its modular flip-graph is $0$ (which is exactly the lower bound provided by Theorem \ref{Atheorem.4.2} when $n=1$).

\addcontentsline{toc}{section}{References}
\bibliographystyle{amsplain}
\bibliography{Torus}

{\em Addresses:}\\
Department of Mathematics, University of Fribourg, Switzerland \\
LIPN, Universit{\'e} Paris 13, Villetaneuse, France\\
{\em Emails:} \href{mailto:hugo.parlier@unifr.ch}{hugo.parlier@unifr.ch}, \href{mailto:lionel.pournin@univ-paris13.fr}{lionel.pournin@univ-paris13.fr}\\

\end{document}